\documentclass[a4paper,11pt]{article}
\usepackage{amsthm, amsmath, amssymb}
\usepackage{setspace} 
\usepackage{enumerate}
\usepackage[shortlabels]{enumitem}
\usepackage{xcolor}
\usepackage{hyperref}
\usepackage{lineno}
\usepackage[colorinlistoftodos]{todonotes}

\newtheorem{theorem}{Theorem}[section]

\newtheorem{cor}[theorem]{Corollary}

\newtheorem{pro}[theorem]{Proposition}

\newcommand{\diag}{\mathop{{ \rm diag}}}

\def \N {{\mathbb N}}

\def \R {{\mathbb R}}

\def \la {{\lambda}}

\author{Priyanka Grover \thanks{The research of this author is supported by a Early Career Research Award ECR/2018/001784 of SERB, India. Email: priyanka.grover@snu.edu.in}, Veer Singh Panwar \thanks{Email: vs728@snu.edu.in}\\
Department of Mathematics\\Shiv Nadar University, Dadri\\ U.P. 201314, India.}
\date{}
  
\begin{document}

%\linenumbers
\title{\textbf{Bidiagonal decompositions and total positivity of some special matrices}} 
\maketitle
 \begin{center}
\large{\textbf{Abstract}}
 \end{center}
\begin{footnotesize}
 The matrix $S = [1+x_i y_j]_{i,j=1}^{n}, 0<x_1<\cdots<x_n,\, 0<y_1<\cdots<y_n$, has gained importance lately due to its role in powers preserving total nonnegativity. We give an explicit decomposition of $S$ in terms of elementary bidiagonal matrices, which is analogous to the \emph{Neville decomposition}. We give a bidiagonal decomposition of $S^{\circ m}=[(1+x_iy_j)^m]$ for positive integers $1\leq m \leq n-1$. We also explore the total positivity of Hadamard powers of another important class of matrices called \emph{mean matrices}. 
\end{footnotesize}

\textbf{AMS classification:} 15B05,15A23,15B48.

\begin{footnotesize}
		\textbf{Keywords :} Totally positive matrices, totally nonnegative matrices, Hadamard powers, infinitely divisible matrices, bidiagonal decomposition, mean matrices.
	\end{footnotesize}

\begin{center}
\section{Introduction}\label{sec:intro}    
\end{center}
A matrix is called \textit{totally nonnegative} (respectively \textit{totally positive}) if all its minors are nonnegative (respectively positive) (see \cite{fallat2011totally}). 
%Let $\M_n(\R)$ be the set of $n\times n$ matrices over $\R.$ 
These matrices have also been called totally positive (respectively strictly totally
positive), as can be seen in \cite{gasca1996factorizations,karlin1964total,pinkus2010totally}.
 Let $A$ %and $B= [B_{ij}]$ 
 be an $n\times n$ matrix with nonnegative entries. The matrix $A^T$ denotes the \emph{transpose} of $A$. We assume $1 \leq i,j \leq n,$   unless otherwise stated. The $(i,j)$th entry of $A$ is denoted by $A_{ij}$.  %The matrix $A$ is said to be{\em positive  semidefinite} ({\em positive definite}) if $A$ is symmetric and $\langle x,Ax \rangle \geq 0$ for all $x \in \R^n$ ($\langle x,Ax \rangle > 0$ for all $x \in \R^n\setminus\{0\}$). 
Let $1\leq k\leq n$. A matrix is called TN$_k$ (respectively TP$_k$) if all its minors upto order $k$ are nonnegative (respectively positive).
 %The {\em Hadamard product} of $A$ and $B,$ denoted by $A\circ B,$ is defined as $A\circ B = [A_{ij}B_{ij}].$
 If $r > 0$, then the $r$th \emph{Hadamard power} of $A$ is given by $A^{or} = [A_{ij}^r].$  The matrix $A$ is said to be {\em infinitely divisible} if $A^{or}$ is positive semidefinite for every $r > 0.$ %It is obvious that every infinitely divisible matrix is positive semidefinite. The converse, however, need not be true (see \cite{bhatia2006infinitely}). Some simple examples of infinitely divisible matrices are nonnegative positive semidefinite matrices of order $2$ and diagonal matrices with nonnegative diagonal entries. 
 We refer the reader to \cite{bhatia2006infinitely,bhatia2007mean,horn1969theory,panwar2022positivity} for many examples and results on infinitely divisible matrices. 

It was shown in \cite{fitzgerald1977fractional} that if $A$ is positive semidefinite, then for $r \geq n-2$, $A^{\circ  r}$ is also positive semidefinite. The sharpness of the lower bound $n-2$ was given by considering the positive semidefinite matrix $A_\epsilon = [1+\epsilon ij]$, where $\epsilon>0$. It was shown that if $r < n-2$ is any positive non integer, then $A_\epsilon ^{\circ r}$ fails to be positive semidefinite for sufficiently small $\epsilon>0$.

A Lebesgue measurable function $f:\R\rightarrow \R$ is called TN${_k}$ if given any increasing sequences $\{x_m\}_{m\geq 1}$ and $\{y_m\}_{m\geq 1}$ of real numbers, the matrix $\left[f(x_i-y_j)\right]$ is TN$_k$. %totally nonnegative for all real numbers $x_1<\cdots<x_r,\, y_1<\cdots<y_r$ and $1\leq r \leq k$.
Let $W:\mathbb R\rightarrow \mathbb R$ be defined as $$W(x)=\begin{cases}\cos x & \textup{if}\,\, x\in (-\pi/2,\pi/2),\\ 
0 & \textup{otherwise}. \end{cases}$$ Schoenberg \cite{1955} showed that if $r\geq 0$ and $k$ is an integer greater than or equal to $2$, then $W(x)^{r}$ is TN$_k$ if and only if $r\geq k-2$. (See also \cite[Remark 6.2]{khare2020multiply}.)
%Recently, \textcolor{red}{Khare} \cite{khare2020multiply} gave a proof for its `only if' part. 

 A function $f:\R\rightarrow \R$ is called a \emph{P\'{o}lya frequency function} if $f$ is Lebesgue integrable on $\R$, does not vanish at at least two points, and  for $n\in \N$ and real numbers $x_1<\cdots<x_n, y_1<\cdots<y_n$, the matrix $\left[f(x_i-y_j)\right]$ is totally nonnegative. 
Consider the P\'{o}lya frequency function $\Omega:\R\rightarrow \R$ defined as $$\Omega(x):=\begin{cases}
xe^{-x} & \text{if}\,\, x>0,\\
0 & \text{otherwise}.
\end{cases}$$ 
Karlin \cite{karlin1964total} proved that for any integer $k\geq 2$ and any real number $r\geq 0$, the function $\Omega(x)^{r}$ is TN$_k$ if $r$ is a non negative integer or $r\geq k-2$, see also \cite[p. 211]{karlin1968total}. Recently, Khare \cite{khare2020multiply} showed its converse by proving that for any integer $k\geq 2$ and $r\in (0,k-2)\setminus \mathbb Z$, $\Omega(x)^{r}$ is not TN$_k$. 

The characterizations of total nonnegativity of Hadamard powers of a matrix have also recently appeared in \cite{fallat2007hadamard,fallat2017total}.
 %Let $0< x_1, \ldots, x_n$. Let $X=[1+x_ix_j]$. Recently, in \cite{jain2017hadamard}, it was proved that $X^{\circ r}$ is positive semidefinite if and only if $r$ is a nonnegative integer or $r>n-2$, see \cite[Theorem 1.1]{jain2017hadamard}. Let $r>n-2$. Then $X^{\circ r}$ is totally positive (see \cite[Remark 2.5]{jain2017hadamard}). Every totally positive symmetric matrix is positive definite. So $X^{\circ r}$ is positive definite.
Let $x_1,\ldots, x_n$ be distinct positive real numbers. Let $X=[1+x_ix_j]$. Jain \cite[Theorem 1.1]{jain2017hadamard} proved that the matrix $X^{\circ r}$ is positive semidefinite if and only if $r$ is a nonnegative integer or $r>n-2$. So the matrix $X$ serves as a stronger example than the matrix $A_\epsilon$ for proving the sharpness of the lower bound $n-2$, in the sense that it works for every positive non integer $r<n-2$. %She also gave the inertia (the number of positive, negative and zero eigenvalues) of the matrix $X^{\circ r}$ for every $r\geq 0$ and distinct $x_i$'s, see \cite[Theorem 2.6]{jain2017hadamard}. 
She also proved that if $0<x_1<\cdots<x_n$, then the matrix $X^{\circ r}$ is totally positive for $r>n-2$ (see \cite[Theorem 2.4]{jain2017hadamard}). 
For any real numbers $x_1<\cdots<x_n$ and $y_1<\cdots<y_n$ such that $1+x_iy_j>0$, let
$$S=[1+x_i y_j].$$
Recently, Khare \cite[Theorem C]{khare2020multiply} showed that the matrix $S^{\circ r}$ is totally positive if $r>n-2$, and totally nonnegative if and only if $r = 0,1,\ldots,n-2$. Jain~\cite[Corollary 5]{jain2020hadamard} showed that for $r = 0,1,\ldots,n-2$, rank$(S^{\circ r})=r+1$, and therefore, $S^{\circ r}$ is not totally positive. Thus $S^{\circ r}$ is totally positive if and only if $r>n-2$ and $S^{\circ r}$ is totally nonnegative if and only if $r>n-2$ or $r = 0,1,\ldots,n-2$. The matrix $S$ was used to prove the converse of Karlin's result, see \cite[Theorem 1.7]{khare2020multiply}. \par
A matrix $A$ is called \emph{lower} (respectively \emph{upper}) \emph{bidiagonal} if $A_{ij}=0$ for $i-j\neq 0,1$ (respectively for $i-j\neq 0,-1$). For any real number $s$ and positive integers $2 \leq i \leq n $, let $L_{i}(s)$ (respectively $U_i(s)$) be the matrix whose diagonal entries are one, $(i,i-1)$th (respectively $(i-1,i)$th) entry is $s$ and the remaining entries are zero. These particular bidiagonal matrices are called \emph{elementary bidiagonal matrices}. Cryer \cite{cryer1976some} showed that any $n\times n$ totally nonnegative matrix $A$ can be written as \begin{equation}
A=\prod_k L^{(k)} \prod_\ell U^{(\ell)}, \label{cryer}
\end{equation} where $L^{(k)}$ and $U^{(\ell)}$ are, respectively, lower and upper elementary bidiagonal matrices %with at most one non zero entry off the main diagonal, i.e. elementary bidiagonal matrices
(see also \cite{fallat2001bidiagonal}). Careful analyses of the relationships between totally nonnegative matrices and bidiagonal decompositions have been done in \cite{fiedler1997consecutive,gasca1996factorizations}. For more results on bidiagonal decompositions of matrices, see \cite{barreras2012bidiagonal,huang2011bidiagonal,huang2013test,johnson1999elementary}.
%Let $E_{i}^{j}$ be the matrix with $(i,j)$th entry equal to $1$ and remaining zero. 
In particular, in the case of invertible totally nonnegative matrices, the unicity
of the bidiagonal decomposition under certain conditions was assured in
\cite{gasca1996factorizations}. Finding explicit decompositions like \eqref{cryer} is a non trivial task as there may not be obvious patterns to guess the factors. One of the main aims of this paper is to give the decomposition \eqref{cryer} for the matrix $S$,  %, where the factors $L^{(k)}$ are products of the elementary bidiagonal matrices in an interesting pattern
which is similar to what appears in the \emph{successive elementary bidiagonal decomposition} (also called \emph{Neville decomposition}) for invertible totally nonnegative matrices, see \cite[Theorem 2.2.2]{fallat2011totally}. However, the matrix $S$ is not invertible for $n\geq 3$. To find this decomposition, we also give an $LU$ decomposition of $S$. % = I+sE_{i}^{i-1},$ where $ 2 \leq i \leq n $. Matrices of the form $ L_{i}(s) $ are called \textit{elementary bidiagonal matrices.}
We also give another interesting decomposition for $S$ in terms of bidiagonal matrices. The difference with the earlier one is that here the lower and upper bidiagonal matrices appear in a mixed pattern; however, a major advantage is that this decomposition can be generalized to Hadamard integer powers of $S$. \par
%In Section \ref{sec2}, we state and prove our results for the decompositions of $S$.
Another important class of matrices is that of the \emph{mean matrices}. For a discussion on infinite divisibility of these matrices, see \cite{bhatia2007mean}. %A binary operation $m$ on the set of positive real numbers is called a \textit{mean} if  $m(a,b)$ is an increasing and continuous function of $a$ and $b$, $m(a,b) = m(b,a)$, $\min(a,b) \leq m(a,b) \leq \text{max}(a,b)$ and for each $\gamma >0$, $m(\gamma a, \gamma b) = \gamma m(a,b)$. 
Some important examples of means on positive real numbers are the \textit{arithmetic mean} $\mathcal{A}(a,b) = \displaystyle \frac{a+b}{2},$ the \textit{harmonic mean} $H(a,b) = \displaystyle \frac{2ab}{a+b},$  the \textit{Heinz mean} $\mathcal{H}_\nu(a,b) = \displaystyle \frac{a^{v}b^{1-v}+a^{1-v}b^{v}}{2}$ for $0\leq \nu \leq 1$ and the \textit{binomial mean} $\mathcal{B}_\alpha(a,b) = \left(\displaystyle\frac{a^{\alpha}+b^{\alpha}}{2}\right)^{1/\alpha}$ for $-\infty \leq \alpha \leq \infty$, where it is understood that $\mathcal{B}_{0}(a,b) = \sqrt{ab}, \mathcal{B}_{\infty}(a,b) = \text{max}(a,b)$ and $\mathcal{B}_{-\infty}(a,b)=\text{min}(a,b)$. 
 Let $0< \la_{1}< \cdots < \la_{n}.$ % For any given mean $m,$ we call the matrix $W = \left[\frac{1}{m(\la_{i},\la_{j})}\right]$ a \textit{mean matrix} corresponding to the mean $m.$ 
 In \cite{bhatia2007mean}, it is shown that $\left[\displaystyle\frac{1}{\mathcal{A}(\la_i,\la_j)}\right],\left[H(\la_i,\la_j)\right],$ $\left[\displaystyle\frac{1}{\mathcal{H}_\nu(\la_i,\la_j)}\right]$ and $\left[\displaystyle\frac{1}{\mathcal{B}_\alpha (\la_i,\la_j)}\right] (\alpha\geq 0)$ are infinitely divisible.  Since the \emph{Cauchy matrix} $C = \left[\displaystyle\frac{1}{\la_{i}+\la_{j}}\right]$ is totally positive (see \cite{pinkus2010totally}), so are $\left[\displaystyle\frac{1}{\mathcal{A}(\la_i,\la_j)}\right]$, $\left[H(\la_i,\la_j)\right]$ and $\left[\displaystyle\frac{1}{\mathcal{H}_\nu(\la_i,\la_j)}\right]$ $\left(\nu\neq \displaystyle\frac{1}{2}\right)$. %and $\left[\displaystyle\frac{1}{\mathcal{B}_\alpha (\la_i,\la_j)}\right]$($0<\alpha<\infty$). 
 In particular, they are TP$_3$.  By \cite[Theorem 4.2]{johnson2012critical} (or \cite[Theorem 5.2]{fallat2017total}), we have that for $r\geq 1$, their $r$th Hadamard powers  are also TP$_3$. We give a simple proof to show that the Hadamard powers of these matrices are in fact totally positive.

In Section \ref{sec2}, we state and prove our results for the decompositions of $S$.
%and $X^{\circ m} (m\in\mathbb N)$. 
In Section \ref{sec3}, we give the results for mean matrices.\\ %Section \ref{sec4} is devoted to some remarks. 
%To prove Theorem \ref{thm: bidiagonal decomposition}, we first give the following decomposition of $S$:

\section{Bidiagonal decompositions for $S=[1+x_iy_j]$ and its Hadamard powers}\label{sec2}
We begin by giving an $LU$ decomposition for $S$.
\begin{pro} \label{LL^* for S}
Let $0<x_1<\cdots<x_n$ and $0<y_1<\cdots<y_n$. Then the matrix $S=[1+x_iy_j]$ can be written as $L U$, where $L$ and $U$ are the lower and upper  triangular matrices, respectively, given by
 \begin{align*}
  L_{ij}&=\begin{cases}
 \displaystyle \frac{1+y_1x_i}{\sqrt{1+x_1y_1}} & \text{if}\,\,j=1,\\[12pt]
  \displaystyle \frac{(x_i-x_1)\sqrt{y_2-y_1}}{\sqrt{x_2-x_1}\sqrt{1+x_1y_1}} & \text{if}\,\,j=2,\\[12pt]
  0 & \text{otherwise}
  \end{cases} \intertext{and}
  U_{ij}&=\begin{cases}
  \displaystyle \frac{1+x_1y_j}{\sqrt{1+x_1y_1}} & \text{if}\,\,i=1,\\[12pt]
  \displaystyle \frac{(y_j-y_1)\sqrt{x_2-x_1}}{\sqrt{y_2-y_1}\sqrt{1+x_1y_1}} & \text{if}\,\,i=2,\\[12pt]
  0 & \text{otherwise}.
  \end{cases}
 \end{align*}
 \end{pro}
 
 \begin{proof}
 This can be proved by checking that the $(i,j)$th entry of $LU$: \begin{align*}
 (LU)_{ij}&= \sum_{k=1}^{2}L_{ik} U_{kj}\\
 &=\left(\frac{1+y_1x_i}{\sqrt{1+x_1y_1}}\right)\left(\frac{1+x_1y_j}{\sqrt{1+x_1y_1}}\right)+\\
 &\quad \left(\frac{(x_i-x_1)\sqrt{y_2-y_1}}{\sqrt{x_2-x_1}\sqrt{1+x_1y_1}}\right)\left(\frac{(y_j-y_1)\sqrt{x_2-x_1}}{\sqrt{y_2-y_1}\sqrt{1+x_1y_1}}\right)\\  
 &= \frac{1}{(1+x_1y_1)} \left[(1+y_1x_i)(1+x_1y_j)+(x_i-x_1)(y_j-y_1)\right]\\ 
 %&= \frac{1}{(1+x_1y_1)}\left[(1+x_iy_j)(1+x_1y_1)\right]\\
 &= 1+x_iy_j.
\end{align*}
\end{proof}

Let $\text{diag}[d_i]$ denote the diagonal matrix with diagonal entries $d_{1},\ldots,d_n.$
Now, we give the decomposition \eqref{cryer} for $S$. %We also note that the decomposition is quite similar to the \emph{successive elementary bidiagonal} (also called \emph{Neville decomposition}) for invertible totally nonnegative matrices, 

%which  characterizes an invertible totally nonnegative matrix in terms of elementary bidiagonal matrices is given in

%see \cite[Theorem 2.2.2]{fallat2011totally}. The matrix $S$, however, is not invertible for $n\geq 3$. 

\begin{theorem}  \label{thm: bidiagonal decomposition}
Let $n\geq 2$. Let $0<x_1<\cdots<x_n$ and $0<y_1<\cdots<y_n$. For $2\leq i\leq n,$ let $\alpha_{i} = \displaystyle \frac{1+y_1x_i}{1+y_1x_{i-1}}$ and $\alpha_{i}^{'}= \dfrac{1+x_1y_i}{1+x_1y_{i-1}}$. For $3\leq j \leq n$, let \begin{align*}
    \beta_{j}&=\dfrac{(x_j-x_{j-1})(1+y_1x_{j-2})}{(x_{j-1}-x_{j-2})(1+y_1x_{j-1})}\intertext{and}
\beta_{j}^{'}&=\frac{(y_j-y_{j-1})(1+x_1y_{j-2})}{(y_{j-1}-y_{j-2})(1+x_1y_{j-1})}. \end{align*} Let $$D = \left[\begin{array}{ccccc}1+x_1y_1 & 0 & 0&\cdots& 0\\ 0& \dfrac{(x_2-x_1)(y_2-y_1)}{1+x_1y_1}&0&\cdots&0\\
0&0&0&\cdots& 0\\
\vdots &\vdots &\vdots &\ddots & \vdots \\
0&0&0&\cdots& 0 \end{array}\right].$$ Then 
\begin{align} \label{thm1.2}
S&=\left(L_n(\alpha_n)\cdots L_2(\alpha_2)\right)\left(L_n(\beta_n)\cdots L_3(\beta_3)\right)D\left(U_3(\beta_{3}^{'})\cdots U_n(\beta_{n}^{'})\right)\nonumber\\
&\quad \left(U_2(\alpha_2^{'})\cdots U_n(\alpha_{n}^{'})\right).
\end{align}

% \begin{bmatrix}1+x_1^2 &  & & & \\&\frac{(x_2-x_1)^2}{1+x_1^2} & & & \\& & 0 & & \\& & & \ddots & \\& & & & 0\end{bmatrix}
\end{theorem}
\begin{proof}
Let the lower triangular matrices $M=[M_{ij}],M^{'}=[M^{'}_{ij}],N=[N_{ij}],N^{'}=[N^{'}_{ij}],Y_1=[(Y_1)_{ij}]$ and $Y_2=[(Y_2)_{ij}]$ be defined as follows:
\begin{align*}
    M_{ij} &= \begin{cases}
    \displaystyle \frac{1+y_1x_i}{1+y_1x_j} & \text{if}\,\, i \geq j,\\[12pt]
    0 & \text{otherwise},
    \end{cases}\\
    M^{'}_{ij} &= \begin{cases}
    \displaystyle \frac{1+x_1y_i}{1+x_1y_j} & \text{if}\,\, i \geq j,\\[12pt]
    0 & \text{otherwise},
    \end{cases}
\end{align*}
\begin{align*}
    N_{ij} &=  \begin{cases}
    1 &\text{if}\,\, i= j=1,\\
    \displaystyle \frac{(x_i-x_{i-1})(1+y_1x_{j-1})}{(x_j-x_{j-1})(1+y_1x_{i-1})} &\text{if}\,\, i \geq j\geq 2,\\[12pt]
    0 & \text{otherwise},
    \end{cases}\\
    N^{'}_{ij} &=  \begin{cases}
    1 &\text{if}\,\, i= j=1,\\
    \displaystyle \frac{(y_i-y_{i-1})(1+x_1y_{j-1})}{(y_j-y_{j-1})(1+x_1y_{i-1})} &\text{if}\,\, i \geq j\geq 2,\\[12pt]
    0 & \text{otherwise},
    \end{cases}\\
(Y_1)_{ij} &= \begin{cases}
 \displaystyle \frac{1+y_1x_i}{1+y_1x_1} &\text{if}\,\,j=1,\\[12pt]
 \displaystyle \frac{x_i-x_{j-1}}{x_j-x_{j-1}} &\text{if}\,\, i\geq j\geq 2,\\[12pt]
 0 &\text{otherwise}
 \end{cases}\\ \intertext{and}
 (Y_2)_{ij} &= \begin{cases}
 \displaystyle \frac{1+x_1y_i}{1+x_1y_1} &\text{if}\,\, j=1,\\[12pt]
 \displaystyle \frac{y_i-y_{j-1}}{y_j-y_{j-1}} &\text{if}\,\, i\geq j\geq 2,\\[12pt]
 0 &\text{otherwise}.
\end{cases}
\end{align*}

Let $P_i = L_i\left(\alpha_{i}\right),P_i^{'} = L_i(\alpha_{i}^{'}),Q_j = L_{j}(\beta_{j})$ and $Q_j^{'} = L_{j}(\beta_{j}^{'})$ for $2\leq i \leq n$ and $3\leq j\leq n$. To prove the theorem, we show the following: 

\begin{enumerate}[(a)]
    \item $P_n P_{n-1}\cdots P_2=M$ and $P_n^{'} P_{n-1}^{'}\cdots P_2^{'}=M^{'}$.  \label{eqn: def of U}
    \item  $Q_{n}Q_{n-1}\cdots Q_3 = N$ and $Q_{n}^{'}Q_{n-1}^{'}\cdots Q_3^{'} = N^{'}$. \label{eqn: def of V}
    \item $MN=Y_1$ and $M^{'}N^{'}= Y_{2}$. \label{eqn: MN=Y}
    \item $S = Y_1DY_2^{T}.$ \label{eqn: X=YDY^T}
\end{enumerate}

We shall give a proof for the first part of each of \ref{eqn: def of U}, \ref{eqn: def of V} and \ref{eqn: MN=Y}. The proofs of their second parts are analogous. So we omit them.

Note that for $i\geq 2,$ multiplying any matrix by $L_{i}(s)$ on the left is equivalent to changing its $i$th row to the one obtained by adding $s$ times the $(i-1)$th row to it. %It will change only the $i$th row of the matrix, and all the other rows will remain the same. 
Let $\delta_{ij} = 1$ for $i=j$, and $0$ otherwise. To prove \ref{eqn: def of U}, we show that for $2\leq k \leq n$,
\begin{align} \label{eqn: induction assumption}
(P_{k}\cdots P_2)_{ij} = \begin{cases}
M_{ij} &\text{if}\,\, i\leq k,\\
\delta_{ij} &\text{if}\,\, i>k.
\end{cases}    
\end{align}

 %We prove this by applying induction on $k.$ 
 Let $I$ be the identity matrix of order $n$. Then for $i\neq 2$, $$(P_2)_{ij}=(P_2I)_{ij}= \delta_{ij}.$$ Note that $M_{1j}=\delta_{1j}$. We also have
\begin{align*}
 (P_2)_{2j} &= (P_2I)_{2j}\\
&= I_{2j}+\left(\frac{1+y_1x_2}{1+y_1x_1}\right)I_{1j}\\
&= \delta_{2j}+\left(\frac{1+y_1x_2}{1+y_1x_1}\right)\delta_{1j}\\
&= \begin{cases}
\displaystyle \frac{1+y_1x_2}{1+y_1x_1} &\text{if}\,\, j=1,\\[10pt]
1 &\text{if}\,\, j=2,\\
0 & \text{otherwise}
\end{cases}\\
&= M_{2j}.  
\end{align*}
Hence (\ref{eqn: induction assumption}) holds for $k=2$. Let it hold for $k = m$, where $2 \leq m \leq n-1$. Then   
\begin{align*}
    (P_{m}\cdots P_2)_{ij} = \begin{cases}
M_{ij} &\text{if}\,\, i\leq m,\\
\delta_{ij} &\text{if}\,\, i> m.
\end{cases}
\end{align*}

\noindent So for $i\leq m,\,
(P_{m+1} \cdots P_2)_{ij} = (P_m \cdots P_2)_{ij} = M_{ij}.$ For $i>m+1$,
 $$(P_{m+1} \cdots P_2)_{ij}\\ = (P_m\cdots P_2)_{ij} = \delta_{ij}.$$
 Also,
\begin{align*}
    (P_{m+1}\cdots P_2)_{m+1,j} &= (P_m \cdots P_2)_{m+1,j} + \left(\frac{1+y_1x_{m+1}}{1+y_1x_m}\right) (P_m \cdots P_2)_{mj}\\
&= \delta_{m+1,j}+\left(\frac{1+y_1x_{m+1}}{1+y_1x_m}\right)M_{mj}\\
%&= \begin{cases}\left(\frac{1+y_1x_{m+1}}{1+y_1x_j}\right) &\text{if}\,\, j\leq m,\\1 &\text{if}\,\, j=m+1,\\0 &\text{otherwise}\end{cases}\\
&= M_{m+1,j}.
\end{align*}
\noindent Hence $$(P_{m+1}\cdots P_2)_{ij} = \begin{cases}
M_{ij} &\text{if}\,\, i\leq m+1,\\
\delta_{ij} &\text{if}\,\, i> m+1.
\end{cases}$$

Thus \eqref{eqn: induction assumption} holds for $k=m+1$. Hence \eqref{eqn: induction assumption} is true for every $2\leq k \leq n.$ Putting $k=n$ proves the first part of \ref{eqn: def of U}. To prove \ref{eqn: def of V}, we show that for $3\leq k\leq n$,
\begin{align} \label{eqn: induction assumption for V}
(Q_{k}\cdots Q_3)_{ij} = \begin{cases}
N_{ij} &\text{if}\,\, i\leq k,\\
\delta_{ij} &\text{if}\,\, i>k.
\end{cases}    
\end{align}

We prove this in a similar manner as above. Since $Q_3 = Q_3I,$ $(Q_3)_{ij}=(Q_3I)_{ij}= \delta_{ij}$ for $i\neq 3.$ Note that for $i=1$ and $i=2$, $N_{ij}=\delta_{ij}$. Now
\begin{align*}
(Q_3)_{3j}&=  I_{3j}+\left(\frac{(x_3-x_2)(1+y_1x_1)}{(x_2-x_1)(1+y_1x_2)}\right)I_{2j}\\
&= \delta_{3j}+\left(\frac{(x_3-x_2)(1+y_1x_1)}{(x_2-x_1)(1+y_1x_2)}\right)\delta_{2j}\\
%&= \begin{cases}\dfrac{(x_3-x_2)(1+y_1x_1)}{(x_2-x_1)(1+y_1x_2)} &\text{if}\,\, j=2,\\1 &\text{if}\,\, j=3,\\0 & \text{otherwise}\end{cases}\\
& = N_{3j}.
\end{align*}
%Hence (\ref{eqn: induction assumption for V}) holds for $k=3.$ 
Suppose \eqref{eqn: induction assumption for V} holds for $k = m$, where $3 \leq m \leq n-1$. Then $$(Q_{m}\cdots Q_3)_{ij} = \begin{cases}
N_{ij} &\text{if}\,\, i\leq m,\\
\delta_{ij} &\text{if}\,\, i> m.
\end{cases}$$
So for $i\leq m$, $$(Q_{m+1} \cdots Q_3)_{ij} = (Q_m \cdots Q_3)_{ij} = N_{ij}.$$ For $i>m+1$, $$(Q_{m+1} \cdots Q_3)_{ij}= \left(Q_m\cdots Q_3\right)_{ij} = \delta_{ij}.$$ 
Also,
\begin{align*}
    (Q_{m+1}\cdots Q_3)_{m+1,j} &= (Q_m \cdots Q_3)_{m+1,j}\\
    &\quad \quad \quad + 
    \left(\frac{(x_{m+1}-x_{m})(1+y_1x_{m-1})}{(x_{m}-x_{m-1})(1+y_1x_m)}\right)(Q_m \cdots Q_3)_{mj}\\
&= \delta_{m+1,j}+\left(\frac{(x_{m+1}-x_{m})(1+y_1x_{m-1})}{(x_{m}-x_{m-1})(1+y_1x_m)}\right)N_{mj}\\
%&= \begin{cases}0 &\text{if}\,\, j=1,\\\dfrac{(x_{m+1}-x_m)(1+y_1x_{j-1})}{(x_j-x_{j-1})(1+y_1x_m)} &\text{if}\,\, 2\leq j\leq m+1,\\0&\text{otherwise}\end{cases}\\
&= N_{m+1,j}.
\end{align*}
Therefore $$(Q_{m+1}\cdots Q_3)_{ij} = \begin{cases}
N_{ij} &\text{if}\,\, i\leq m+1,\\
\delta_{ij} &\text{if}\,\, i> m+1.
\end{cases}$$
So \eqref{eqn: induction assumption for V} is true for every $3\leq k \leq n$. Putting $k=n$ proves the first part of \ref{eqn: def of V}. Now for each $i$,
\begin{align*}
     (MN)_{i1}%= \sum_{k=1}^{n}M_{ik}N_{k1}
     =M_{i1}N_{11}= \left(\frac{1+y_1x_i}{1+y_1x_1}\right)=(Y_{1})_{i1}. 
 \end{align*}
 For $i<j$,
 \begin{align*}
     (MN)_{ij}%= \sum_{k=1}^{n}M_{ik}N_{kj} 
     =0=(Y_{1})_{ij}.
 \end{align*}
For $1 < j \leq i,$
    \begin{align*}
    (MN)_{ij} &= \sum_{k=1}^{n}M_{ik}N_{kj}\\
    &= \sum_{i\geq k \geq j}\left(\frac{1+y_1x_i}{1+y_1x_k}\right)\left(\frac{(x_k-x_{k-1})(1+y_1x_{j-1})}{(x_j-x_{j-1})(1+y_1x_{k-1})}\right)\\
    &=\frac{(1+y_1x_i)(1+y_1x_{j-1})}{y_1(x_j-x_{j-1})}\sum_{i\geq k\geq j}\left[\frac{1}{(1+y_1x_{k-1})}-\frac{1}{(1+y_1x_{k})}\right]\\
    &=\frac{(1+y_1x_i)(1+y_1x_{j-1})}{y_1(x_j-x_{j-1})}\left[\frac{1}{(1+y_1x_{j-1})}-\frac{1}{(1+y_1x_{i})}\right]\\
    &=\frac{x_i-x_{j-1}}{x_j-x_{j-1}}\\
    &=(Y_{1})_{ij}.
    \end{align*}
    This proves the first part of \ref{eqn: MN=Y}.
    
Let $\sqrt{D}= \text{diag}[\sqrt{D_{ii}}].$ Let $L$ and $U$ be the lower and upper triangular matrices, respectively, in Proposition \ref{LL^* for S}. To prove \ref{eqn: X=YDY^T}, we note that $Y_1\sqrt{D}= L$ and $\sqrt{D}Y_2^T=U$. Since $LU=S$, we have $Y_1DY_2^{T}=(Y_1\sqrt{D})(\sqrt{D}Y_2^{T}
)=S$. This completes our proof.
\end{proof}

In particular, if $x_i = y_i =i$ for $1\leq i\leq n$, then we have the following corollary.

\begin{cor} We have \begin{equation}X = \left[1+ij\right]= ZDZ^T,\label{eqn2}\end{equation} where 
\begin{align*}
Z &= \left[L_n\left(\frac{n+1}{n}\right)L_{n-1}\left(\frac{n}{n-1}\right)\cdots L_2\left(\frac{3}{2}\right)\right]\left[L_n\left(\frac{n-1}{n}\right)L_{n-1}\left(\frac{n-2}{n-1}\right)\right.\\
& \quad\left.\cdots L_3\left(\frac{2}{3}\right)\right] \text{ and } D =\left[\begin{array}{ccccc}2 & 0 & 0 &\ldots & 0\\
0& \frac{1}{2}& 0 &\ldots& 0\\
\ &\ &\ddots &\ &\ \\
0 & 0 & \ldots& \ldots & 0\end{array}\right].
\end{align*}
%, where $d_1 = 2, d_2=\frac{1}{2}$ and $d_i=0$ for $i\geq 3$.

%\begin{bmatrix}2 &  & & & \\&\frac{1}{2} & & & \\& & 0 & & \\& & & \ddots & \\& & & & \end{bmatrix}
\end{cor}

In the next theorem, we give a bidiagonal decomposition of the $m$th Hadamard powers of $S$ for $m \in \{1,2, \ldots, n-1\}$. For distinct real numbers $x_1,\ldots,x_n$ and $1\leq k \leq n-1$, let the lower bidiagonal matrices $L^{\mathbf x(k)}$ and upper bidiagonal matrices $U^{\mathbf x(k)}$ be defined as
\begin{align*}
(L^{\mathbf{x}(k)})_{ij}&=\begin{cases}
    1 & \text{if}\,\, i=j,\\
    1 & \text{if}\,\, i=j+1,\,i= n-k+1,\\
    \prod\limits_{t=0}^{k-n+i-2}\displaystyle \frac{x_{i}-x_{i-1-t}}{x_{i-1}-x_{i-2-t}} & \text{if}\,\, i=j+1,\,i> n-k+1,\\
    0 & \text{otherwise},
\end{cases}
\end{align*}
and \begin{align*}
(U^{\mathbf{x}(k)})_{ij}&=\begin{cases}
    1 & \text{if}\,\, i=j,\,i\leq n-k,\\
    x_{i}-x_{n-k} & \text{if}\,\, i=j,\,i> n-k,\\
    x_1 & \text{if}\,\, i=j-1, \,i= n-k,\\
    x_{k-n+i+1}\prod\limits_{t=1}^{k-n+i} \displaystyle \frac{x_{i}-x_{i-t}}{x_{i+1}-x_{i+1-t}} & \text{if}\,\, i=j-1, \,i> n-k,\\
    0 & \text{otherwise}.
\end{cases}
\end{align*}

\begin{theorem} \label{thm: bid decomp Hdmrd pow}
Let $n\geq 2$. Let $m\in \{1,\ldots, n-1\}$. Let $x_1,\ldots,x_n$ and $y_1,\ldots, y_n$ be distinct real numbers. Let $$D_m=\left[\begin{array}{ccccccc}1 &\ &\ &\ &\ &\ & \ \\ \ & \binom{m}{1} & \ & \ & \ &\ &\ \\ \ &\ & \binom{m}{2}& \ &\ &\ & \ \\ \ &\ & \ & \ddots &\ &\ & \ \\
\ &\ & \ & \ &\binom{m}{m-1}  &\ & \ \\ \ &\ & \ & \ &\  &1 & \ \\ \ &\ & \ & \ &\  & &\boldsymbol 0_{n-m-1}\end{array}\right],$$ where $\binom{m}{i}$ denotes the binomial coefficient and $\boldsymbol 0_{n-m-1}$ is the zero matrix of order $n-m-1$. 
Then 
\begin{align*}
S^{\circ m} &= \left(L^{\mathbf{x}(1)}\cdots L^{\mathbf{x}(n-1)}U^{\mathbf{x}(n-1)}\cdots U^{\mathbf{x}(1)}\right)D_m\left(L^{\mathbf{y}(1)}\cdots L^{\mathbf{y}(n-1)}U^{\mathbf{y}(n-1)}\cdots \right. 
\\ 
& \left. \quad \,\, U^{\mathbf{y}(1)}\right)^{T}.
\end{align*}
\end{theorem}

\begin{proof}
Let $V_\mathbf{x}$ be the Vandermonde matrix given by
$$V_\mathbf{x}=\left[x_{i}^{j-1}\right]=\begin{bmatrix}
    1 & x_1 & \cdots &x_1^{n-1}\\
    1 & x_2 & \cdots & x_2^{n-1}\\
    \vdots & \vdots & \ddots & \vdots \\
    1&x_n&\cdots & x_n^{n-1}
    \end{bmatrix}$$ and let $V_\mathbf{y}$ be defined analogously. 
We note that
\begin{align*}
    (V_\mathbf{x}D_m V_\mathbf{y}^{T})_{ij}&=\sum_{k=1}^{n}(V_\mathbf{x})_{ik}(D_m V_{\mathbf{y}}^{T})_{kj}\\
    &=\sum_{k=1}^{n}x_{i}^{k-1}\binom{m}{k-1}y_{j}^{k-1}\\
    &=\sum_{k=0}^{m}\binom{m}{k}x_{i}^{k}y_{j}^{k}\\
    &=(1+x_iy_j)^{m}.
\end{align*}
Thus 
\begin{equation} \label{eqn: def of Vandermonde}
    S^{\circ m}= V_\mathbf{x}D_m V_\mathbf{y}^{T}.
\end{equation}
Further, by \cite[Theorem 3.1]{orucc2000explicit}, the Vandermonde matrices $V_x$ and $V_{\mathbf{y}}$ can be factorized as follows:
\begin{equation}V_\mathbf{x} = L^{\mathbf{x}(1)}L^{\mathbf{x}(2)}\cdots L^{\mathbf{x}(n-1)}U^{\mathbf{x}(n-1)}U^{\mathbf{x}(n-2)}\cdots U^{\mathbf{x}(1)}\label{xx}\end{equation}
and
\begin{equation}V_\mathbf{y} = L^{\mathbf{y}(1)}L^{\mathbf{y}(2)}\cdots L^{\mathbf{y}(n-1)}U^{\mathbf{y}(n-1)}U^{\mathbf{y}(n-2)}\cdots U^{\mathbf{y}(1)}.\label{xxx}\end{equation}
Substituting \eqref{xx} and \eqref{xxx} in \eqref{eqn: def of Vandermonde}, we get the desired result. 
\end{proof}

%In Section $\ref{section: Bidiagonal decomposition}$, we give the bidiagonal decompositions for the matrices $S$ and its Hadamard powers. In Section \ref{sec: mean matrices}, we prove the above result for mean matrices. Section $\ref{sec:  remarks}$ is devoted to some remarks.

%\section{Mean matrices} \label{section: total positivity}Let $0<\la_1< \cdots <\la_n.$ We first consider the mean matrices %$\left[\bigg(\frac{1}{\text{A}(\la_i,\la_j)}\bigg)^r \right]$, $\left[\bigg(\frac{1}{\text{H}(\la_i,\la_j)}\bigg)^r \right]$, $\left[\bigg(\frac{1}{H_\nu(\la_i,\la_j)} \bigg)^r \right]$ and $\left[\bigg(\frac{1}{B_\alpha (\la_i,\la_j)}\bigg)^r \right],$ where $\alpha,r>0, \, 0\leq \nu \leq 1,\, \nu \neq \frac{1}{2}$ and

%whose $(i,j)$th entries are \begin{align*} \frac{1}{\mathcal{A}(\la_i,\la_j)} &= \frac{2}{\la_i + \la_j},\\ \frac{1}{\mathcal{H}(\la_i,\la_j)} &= \frac{2\la_i\la_j}{\la_i+\la_j},\\ \frac{1}{\mathcal{H}_\nu(\la_i,\la_j)} &= \frac{2}{\la_i^{\nu}(\la_i^{1-2\nu}+\la_j^{1-2\nu})\la_j^{\nu}} \,\,\text{for}\,\, 0\leq \nu \leq 1, \nu \neq \frac{1}{2}\\ \intertext{and    \frac{1}{\mathcal{B}_\alpha (\la_i,\la_j)} &=   \left(\frac{2}{\la_i^{\alpha}+\la_j^{\alpha}}\right)^{\frac{1}{\alpha}},\,\text{where}\,\, \alpha >0.\end{align*}
 
\section{Mean matrices} \label{sec3}

 %This section explores the total positivity of Hadamard powers of another important set of matrices, namely {the mean matrices}. 
 We first note the below easy proposition about the Cauchy matrix $C = \left[\displaystyle\frac{1}{\la_{i}+\la_{j}}\right]$, where $0<\la_1<\cdots<\la_n$ are positive real numbers. The matrix $C$ is known to be infinitely divisible (see \cite{bhatia2006infinitely}). %and totally positive \cite[page 92]{pinkus2010totally}.
\begin{pro} \label{Cauchy hadamard is TP}
For $r>0$, $C^{\circ r}$ is totally positive.
 \end{pro}
\begin{proof}
Let $r>0$. Every minor of $C^{\circ r}$ is of the form $\det\left(\left[\displaystyle\frac{1}{(p_{i}+q_j)^r}\right]\right),$ where $0<p_1 < \cdots <p_n$ and $0<q_1< \cdots < q_n$. We have $\displaystyle\frac{1}{\left(p_i+q_j\right)^r} = \displaystyle\frac{1}{p_i^r}\displaystyle\frac{1}{\big(1+(q_j/p_i)\big)^r}$. % Let $D = \text{diag}\left[\frac{1}{p_i}\right].$ Then $\left[\frac{1}{(p_i+q_j)^r}\right] = D^{\circ r}\left[\frac{1}{\big(1+\frac{q_j}{p_i}\big)^r}\right]$. 
 Since $\left[\displaystyle\frac{1}{\big(1+(q_j/p_i)\big)^r}\right]$ is nonsingular (see \cite[Corollary 5]{jain2020hadamard}), so is $\left[\displaystyle\frac{1}{(p_i+q_j)^r}\right]$. Therefore, $\det \left(\left[\displaystyle\frac{1}{(p_i+q_j)^r}\right]\right)\neq 0$ for every $r>0$  and for every $0<p_1<\cdots<p_n, 0<q_1<\cdots<q_n$. The map $(p_1,\ldots,p_n,q_1,\ldots,q_n,r)\mapsto \det \left(\left[\displaystyle\frac{1}{(p_i+q_j)^r}\right]\right)$ is a continuous function of its variables. So by the intermediate value theorem, 
 it retains its sign for all choices of $0<p_1< \cdots < p_n, 0<q_1 < \cdots < q_n$ and $r>0.$ For $r=1,p_i = i$ and $q_j = j,$ we have, $\det \left(\left[\displaystyle\frac{1}{i+j}\right]\right)>0$ (see \cite[p. $92$]{pinkus2010totally}). Thus $\det\left(\left[\displaystyle\frac{1}{(p_{i}+q_j)^r}\right]\right)>0$. So $C^{\circ r}$ is totally positive.
\end{proof}
\noindent We remark that the total positivity of Hadamard powers of  Pascal matrices is shown in \cite[Remark 2.2]{grover2020positivity}.\\
\par

The main theorem of this section is as below. The proof is similar to \cite{bhatia2007mean}, where their infinite divisibility is discussed.

\begin{theorem} \label{thm:mean}
Let $r>0$. The matrices
 $\left[\displaystyle\frac{1}{\mathcal{A}(\la_i,\la_j)^r}\right],\left[H(\la_i,\la_j)^r\right],\\\left[\displaystyle\frac{1}{\mathcal{H}_\nu(\la_i,\la_j)^r}\right]$ ($\nu\neq \frac{1}{2}$) and $\left[\displaystyle\frac{1}{\mathcal{B}_\alpha (\la_i,\la_j)^r}\right]$($0<\alpha<\infty$)  are totally positive. The matrices
    $\left[\displaystyle\frac{1}{\mathcal{H}_{\frac{1}{2}}(\la_i,\la_j)^r}\right]$ $\left(= \left[\displaystyle\frac{1}{\mathcal{B}_0 (\la_i,\la_j)^r}\right]\right)$ and $\left[\displaystyle\frac{1}{\mathcal{B}_\infty (\la_i,\la_j)^r}\right]$ are totally nonnegative.
\end{theorem}

 %For  let $F_{\alpha} =  \left[\frac{1}{\mathcal{F}_\alpha (\la_{i},\la_{j})}\right].$ It is known that $F_{\alpha}$ is infinitely divisible for $\frac{1}{2}\leq \alpha \leq 1$ (see \cite{bhatia2007mean}). We show that  \parTo prove Theorem \ref{thm:mean},
%Alternatively, first we see that the matrix $\left[\frac{1}{(\la_i+\mu_j)^r}\right]$ is nonsingular for all values of \textcolor{red}{positive real numbers} $\la_1 < \ldots < \la_n , \mu_1 < \ldots < \mu_n$ and for $r>0.$ Suppose it is singular for some $r>0$ and \textcolor{red}{positive real numbers} $\la_1 < \cdots < \la_n , \mu_1 < \cdots < \mu_n,$ then there exists real numbers $c_1,\ldots, c_n,$ not simultaneously zeros, such that\begin{align*}\sum_{j=1}^{n} \frac{c_j}{(\la_i+\mu_j)^r}=0  \quad \quad \text{for} \, \, i = 1,\ldots, n.\end{align*} This gives that either the function $f_r(x): = \sum_{j=1}^{n} \frac{c_j}{(x+\mu_j)^r}$ is identically zero \textcolor{red}{(which is not, why?)} or has at least $n$ zeros in $(0,\infty),$ which is not possible by \cite[Proposition $2.3$]{jain2017hadamard}. 

\begin{proof}
Since $\left[\displaystyle\frac{1}{\mathcal{A}(\la_i,\la_j)}\right]$ is a Cauchy matrix, the total positivity of its Hadamard powers follows from Proposition \ref{Cauchy hadamard is TP}. The total positivity and total nonnegativity of a matrix are preserved under multiplication by a diagonal matrix with positive diagonal entries. Note that \begin{align*}
\left[H(\la_i,\la_j)\right] &= \text{diag}\left[\sqrt{2}\la_i\right]\left[\displaystyle\frac{1}{\la_i+\la_j}\right]\text{diag}\left[\sqrt{2}\la_i\right]
\intertext{and}
\left[\frac{1}{\mathcal{B}_\alpha(\la_i,\la_j)}\right] &= \text{diag}\left[2^{1/\alpha}\right]\left[\displaystyle\frac{1}{(\la_i^{\alpha}+\la_j^{\alpha})^{1/\alpha}}\right].
\end{align*} Thus by Proposition \ref{Cauchy hadamard is TP}, $\left[H(\la_i,\la_j)^r\right]$ is totally positive and so is $\left[\displaystyle\frac{1}{\mathcal{B}_\alpha(\la_i,\la_j)^r}\right]$ for $0<\alpha<\infty$. Similarly, since we have 
\begin{align*}
    \left[\displaystyle\frac{1}{\mathcal{H}_\nu(\la_i,\la_j)}\right] &=\diag \left[\displaystyle\frac{1}{\la_i^{\nu}}\right] \left[\displaystyle\frac{2}{\la_i^{1-2\nu}+\la_j^{1-2\nu}}\right] \diag \left[\displaystyle\frac{1}{\la_i^{\nu}}\right],
\end{align*}
we get that $\left[\displaystyle\frac{1}{\mathcal{H}_\nu(\la_i,\la_j)^r}\right]$ is totally positive for $0\leq \nu<\displaystyle\frac{1}{2}$. Since
$\mathcal{H}_{\nu}$ is symmetric about $\nu=\displaystyle\frac{1}{2}$, the Hadamard powers of $\left[\displaystyle\frac{1}{\mathcal{H}_\nu(\la_i,\la_j)}\right]$ are also totally positive for $\displaystyle\frac{1}{2}< \nu \leq 1$.  
%For $\nu=1/2$ or $\alpha=0$,
Also, 
$$\left[\displaystyle\frac{1}{\mathcal{H}_{\frac{1}{2}}(\la_i,\la_j)^r}\right] =\left[\displaystyle\frac{1}{\mathcal{B}_0 (\la_i,\la_j)^r}\right] = \text{diag}\left[\displaystyle\frac{1}{(\sqrt{\la_i})^r}\right]F\,\text{diag}\left[\displaystyle\frac{1}{(\sqrt{\la_i})^r}\right],$$ where $F$ is the \emph{flat matrix} with all its entries equal to $1$. Thus it is totally nonnegative. The matrix $$\left[\displaystyle\frac{1}{\mathcal{B}_\infty (\la_i,\la_j)^r}\right] = \left[\displaystyle\frac{1}{\text{max}(\la_{i}^r,\la_{j}^r)}\right] = \left[\text{min}\left(\displaystyle\frac{1}{\la_i^r},\frac{1}{\la_j^r}\right)\right].$$ So the $(i,j)$th entry of $\left[\displaystyle\frac{1}{\mathcal{B}_\infty (\la_i,\la_j)^r}\right]$ is the $(n+1-i,n+1-j)$th entry of $\left[\text{min}\left(\displaystyle\frac{1}{\la_{n+1-i}^r},\frac{1}{\la_{n+1-j}^r}\right)\right]$. %Since $0<\frac{1}{\la_{n}^r}<\frac{1}{\la_{n-1}^r}<\cdots<\frac{1}{\la_{1}^r}$, 
In view of Proposition 1.3 of \cite{pinkus2010totally}, it is enough to show that if $0<\mu_1<\cdots<\mu_n$, then $\left[\min(\mu_i,\mu_j)\right]$ is totally nonnegative. 
\noindent To see this, let $L^{'}=[L^{'}_{ij}]$ and $U^{'}=[U^{'}_{ij}]$ be defined as 

%A decomposition for the matrix $[\text{min}(i,j)]$ is given on page $17$ of \cite{bhatia2006infinitely}, which shows that it is totally nonnegative. A generalization of this decomposition for $0<\la_1<\cdots < \la_n$ can be given as 

 $$L^{'}_{ij} = \begin{cases}\mu_1 &\text{if}\,\, i\geq j=1,\\(\mu_j-\mu_{j-1}) &\text{if}\,\, i\geq j\geq 2,\\0 &\text{otherwise}\end{cases}$$ and $$U^{'}_{ij} = \begin{cases}1 &\text{if}\,\, i\leq j,\\0 &\text{otherwise}.\end{cases}$$
Then $L^{'}$ and $U^{'}$ are totally nonnegative lower and upper triangular matrices, respectively, and $$[\text{min}(\mu_i,\mu_j)]=L^{'}U^{'}.$$

\noindent (For $\mu_i=i$, this decomposition is given in \cite{bhatia2006infinitely}.) Thus, $\left[\min(\mu_i,\mu_j)\right]$ is totally nonnegative. %Thus the matrix $[\text{min}(\mu_i,\mu_j)]$ is also totally nonnegative. Let $J$ be the exchange matrix of order $n$ given by $J_{ij}= 1$ if $j=n-i+1$ and $0$ otherwise. If $A$ is any totally nonnegative matrix of order $n$ then so is $JAJ$. Hence the matrix $[\text{min}(\la_i,\la_j)]$ is totally nonnegative even if $\la_i$'s are in decreasing order. Since the matrix $\left[\frac{1}{\mathcal{B}_\infty (\la_i,\la_j)^r}\right] = \left[\frac{1}{\text{max}(\la_{i}^r,\la_{j}^r)}\right] = \left[\text{min}\left(\frac{1}{\la_i^r},\frac{1}{\la_j^r}\right)\right]$, it is totally nonnegative.
\end{proof} 

\section{Remarks}\label{remarks} 
%Theorem \ref{thm:mean} is stronger than this.
%The function $F_{\alpha}(a,b) = (1-\alpha) \sqrt{ab}+ \alpha \big(\frac{a+b}{2}\big),$ where $0\leq \alpha \leq 1,$ is called the \emph{Heron mean}. For positive real numbers $p_1,\ldots , p_n,$ 

%\begin{pro}  For $t>0$ and $r>0,$ $H_{t}^{\circ r}$ is $\text{TP}_{2}.$\end{pro}
\begin{enumerate}
    \item Note that for $0<\alpha < \infty$, $\mathcal{B}_{-\alpha}(a,b) = \displaystyle\frac{ab}{\mathcal{B}_{\alpha}(a,b)}$. Thus $\left[\mathcal{B}_{-\alpha}(\la_i,\la_j)^r\right]$ is totally positive for $0< \alpha < \infty$. Also, $\left[\mathcal{B}_{-\infty}(\la_i,\la_j)^r\right]=[\textup{min}(\la_i^r,\la_j^r)]$ is totally nonnegative.
    \item In Theorem \ref{thm: bidiagonal decomposition}, the decomposition holds for all real numbers $x_1,\ldots,x_n$ and $y_1,\ldots, y_n$ such that $1+x_1y_j,1+x_iy_1$ are nonzero for $1\leq i,j\leq n$, and $x_i-x_{i-1}, y_i-y_{i-1}$ are nonzero for $2\leq i \leq n$. In particular, it holds for all distinct positive real numbers.
\end{enumerate} %So $\left[\mathcal{B}_{-\alpha}(\la_i,\la_j)\right]$

\vspace{0.5cm}

\textbf{Acknowledgement:} \textit{It is a pleasure to record our sincere thanks to Apoorva Khare for his insightful comments that helped to improve our results. We are thankful to Shaun M. Fallat for directing us to the preprint \cite{khare2020multiply}.}

\end{document}